\input ./style/arxiv-vmsta.cfg
\documentclass[numbers,compress,v1.0.1]{vmsta}
\usepackage{vtexbibtags}
\usepackage{mathbh}

\volume{6}
\issue{2}
\pubyear{2019}
\firstpage{145}
\lastpage{165}
\aid{VMSTA133}
\doi{10.15559/19-VMSTA133}
\articletype{research-article}

\startlocaldefs
\newcommand{\sparse}[1]{sp\tsup{#1}}

\numberwithin{equation}{section}
\newtheorem{prop}{Proposition}[section]
\newtheorem{theo}[prop]{Theorem}

\theoremstyle{definition}
\newtheorem{ex}[prop]{Example}
\newtheorem{definition}{Definition}

\newcommand{\rrVert}{\Vert}
\newcommand{\llVert}{\Vert}
\newcommand{\rrvert}{\vert}
\newcommand{\llvert}{\vert}
\def\tsup#1{\textsuperscript{#1}}
\allowdisplaybreaks
\endlocaldefs

\begin{document}

\begin{frontmatter}
\pretitle{Research Article}

\title{The asymptotic error of chaos expansion approximations for
stochastic differential equations}

\author[a]{\inits{T.}\fnms{Tony}~\snm{Huschto}\ead[label=e1]{tony.huschto@roche.com}}
\author[b]{\inits{M.}\fnms{Mark}~\snm{Podolskij}\thanksref{cor1}\ead[label=e2]{mpodolskij@math.au.dk}}
\thankstext[type=corresp,id=cor1]{Corresponding author.}
\author[c]{\inits{S.}\fnms{Sebastian}~\snm{Sager}\ead[label=e3]{sager@ovgu.de}}

\address[a]{Department of Mathematics,
\institution{Heidelberg University},
Im Neuenheimer Feld 205, 69120 Heidelberg,
\cny{Germany}}
\address[b]{Department of Mathematics,
\institution{Aarhus University},
Ny Munkegade 118,\break 8000 Aarhus,
\cny{Denmark}}
\address[c]{Faculty of Mathematics,
\institution{Otto-von-Guericke Universit\"{a}t Magdeburg},
Universit\"{a}tsplatz 2, 39106 Magdeburg,
\cny{Germany}}



\markboth{T. Huschto et al.}{The asymptotic error of chaos expansion
approximations for stochastic differential equations}

\begin{abstract}
In this paper we present a numerical scheme for stochastic differential
equations based upon the Wiener chaos expansion. The approximation of
a square integrable stochastic differential equation is obtained by
cutting off the infinite chaos expansion in chaos order and in number
of basis elements. We derive an explicit upper bound for the
$L^{2}$ approximation error associated with our method. The proofs are
based upon an application of Malliavin calculus.
\end{abstract}

\received{\sday{22} \smonth{11} \syear{2018}}
\revised{\sday{9} \smonth{3} \syear{2019}}
\accepted{\sday{9} \smonth{3} \syear{2019}}
\publishedonline{\sday{23} \smonth{4} \syear{2019}}
\end{frontmatter}

\section{Introduction}%
\label{sec1}

We consider a one-dimensional continuous stochastic process
$(X_{t})_{t \in [0,T]}$ that satisfies the stochastic differential
equation
%
\begin{align}
\label{sde} dX_{t}=b(t,X_{t}) dt+ \sigma (t,
X_{t}) dW_{t} \quad \text{with} \
X_{0}=x_{0},
\end{align}
where $(W_{t})_{t \in [0,T]}$ is a Brownian motion 
defined on
a filtered probability space $(\varOmega , \mathcal{F},\allowbreak (
\mathcal{F}_{t})_{t \in [0,T]}, \mathbh{P}   )$. Various numerical
approximation schemes for the SDE
\eqref{sde} have been proposed and
studied in the literature in the past decades. The probably most
prominent numerical approximation for the solution of \eqref{sde} is the
Euler scheme, which can be described as follows. Let $\varphi _{n}:
\mathbh{R}_{+} \to \mathbh{R}_{+}$ be the function defined by
$\varphi _{n}(t)=i/n$ when $t \in [i/n, (i+1)/n)$. The continuous Euler
approximation scheme\index{Euler approximation scheme} is described by
%
\begin{align}
\label{Eulersde} dX_{t}^{n}=b \bigl(\varphi
_{n} (t), X_{\varphi _{n} (t)}^{n} \bigr) dt+ \sigma \bigl(
\varphi _{n} (t), X_{\varphi _{n} (t)}^{n}
\bigr)dW_{t} \quad \text{with} \ X_{0}^{n}=x_{0}.
\end{align}
The probabilistic properties\index{probabilistic properties} of the Euler approximation scheme\index{Euler approximation scheme} have been
investigated in numerous papers. We refer to the classical works
\cite{BT1,BT2,JP,KP,KuPr,MP} for the studies on weak and strong
approximation errors among many others. Asymptotic results in the
framework of non-regular coefficients can be found in e.g.
\cite{A,CS,HK,Yan}.

In this paper we take a different route and propose to use the Wiener chaos
expansion\index{Wiener chaos expansion} (also called polynomial chaos in the literature) to
approximate the solution of the SDE 
\eqref{sde}. To explain ideas let
us fix an orthonormal basis\index{orthonormal basis} $(e_{i})_{i\geq 1}$ of the separable Hilbert
space $L^{2}([0,T])$. It is a well-known statement (cf. \cite{CM47})
that if $X_{t} \in L^{2}(\varOmega , \mathcal{F}, \mathbh{P})$ for all
$t \in [0,T]$, where $\mathcal{F}$ is generated by the Brownian motion
$(W_{t})_{t \in [0,T]}$, it admits the chaotic expansion
%
\begin{align}
\label{choasx} X_{t} = \sum_{\alpha \in \mathcal{I}}
x_{\alpha } (t) \varPsi ^{\alpha },
\end{align}
where $x_{\alpha }$ are deterministic functions, $\varPsi ^{\alpha }$ are
mutually orthogonal random projections that are associated to the basis
$(e_{i})_{i\geq 1}$, and the index set $\mathcal{I}$ is defined via
%
\begin{align}
\label{setI} \mathcal{I}:= \bigl\{\alpha =(\alpha _{i})_{i \geq 1}:~
\alpha _{i} \in \mathbh{N}_{0} \text{ and almost all $
\alpha _{i}$'s are } 0 \bigr\}.
\end{align}
Such orthogonal expansions have been successfully applied in numerous
fields of stochastic and numerical analysis. We refer e.g. to
\cite{HS13,HS14,IS83,VK76} for applications of the Wiener chaos expansion\index{Wiener chaos expansion}
in the context of SDEs and to \cite{LR06,LMR97,L06,WRK09} for
applications of polynomial expansion in modelling, simulation and
filtering of stochastic partial differential equations. The aim of this
work is to use the chaos expansion\index{chaos expansion}
\eqref{choasx} to numerically
approximate the solution of the SDE 
\eqref{sde}. For this purpose we
need to truncate the infinite sum in \eqref{choasx}. An obvious approach
is to consider the approximation
%
\begin{align}
\label{xpk} X_{t}^{p,k} = \sum
_{\alpha \in \mathcal{I}_{p,k}} x_{\alpha } (t) \varPsi ^{\alpha },
\end{align}
where the subset $\mathcal{I}_{p,k} \subset \mathcal{I}$ refers to using
orthogonal projections $\varPsi ^{\alpha }$ only with respect to the first
$k$ basis elements $(e_{i})_{1\leq i \leq k}$ and only up to the $p$th
order Wiener chaos.\index{Wiener chaos} This method is mostly related to the articles
\cite{HS13,HS14,L06,LMR97}. More specifically, the works
\cite{LMR97,L06} study the $L^{2}$-error associated with the
approximation 
\eqref{xpk}, but only for a particular choice of the
basis $(e_{i})_{i \geq 1}$. In this paper we will study the decay rate
of $\mathbh{E}[(X_{t}-X_{t}^{p,k})^{2}]$ when $k,p \to \infty $ for a
general basis $(e_{i})_{i \geq 1}$ of $L^{2}([0,T])$ applying methods
from Malliavin calculus.\index{Malliavin calculus}

The paper is structured as follows. In Section \ref{sec2} we present the
elements of Malliavin calculus.\index{Malliavin calculus} The main results of the paper are
demonstrated in Section \ref{sec3}. Section \ref{sec4} is devoted to
proofs. In Section~\ref{sec5} we illustrate our approach with exemplary
numerical results for the Haar and a trigonometric basis,\index{trigonometric basis} and propose
a heuristic based on sparse indices for computational speedup.

\section{Background on Malliavin calculus\index{Malliavin calculus}}%
\label{sec2}

In this section we introduce some basic concepts of Malliavin calculus.\index{Malliavin calculus}
The interested readers are referred to \cite{N} for more thorough
exposition on this subject. Set $\mathbh{H} = L^{2}([0,T])$ and let
$\langle \cdot , \cdot \rangle _{\mathbh{H}}$ denote the scalar product
on $\mathbh{H}$. We note that $\mathbh{H}$ is a separable Hilbert space
and denote by $(e_{i})_{i \geq 1}$ an orthonormal basis\index{orthonormal basis} of $
\mathbh{H}$. We consider the \textit{isonormal Gaussian family}
$W =\{ W(h) : h \in \mathbh{H} \}$ indexed by $\mathbh{H}$ defined on
a probability space $(\varOmega , \mathcal{F}, \mathbh{P})$, i.e. the
random variables $W(h)$ are centered Gaussian with a covariance
structure determined via
\begin{equation*}
\mathbh{E}\bigl[ W(g) W(h) \bigr] = \langle g,h \rangle _{\mathbh{H}}.
\end{equation*}
Here and throughout the paper we assume that $\mathcal{F}= \sigma (W)$.
In our setting we consider $W(h) = \int_{0}^{T} h_{s} dW_{s}$, where
$W$ is a standard Brownian motion. We define the normalised Hermite
polynomials through the identities
%
\begin{align}
\label{hermite} H_{0}(x):=1, \qquad H_{n}(x):=
\frac{(-1)^{n}}{\sqrt{n!}} \exp \biggl( \frac{x^{2}}{2} \biggr) \frac{\mathrm{d}^{n}}{\mathrm{d} x^{n}}
\biggl( \exp \biggl( -\frac{x^{2}}{2} \biggr) \xch{\biggr),}{\biggr)} \quad n \geq 1.
\end{align}
The $n$\textit{th Wiener chaos\index{Wiener chaos}} $\mathcal{H}_{n}$ is the closed linear
subspace of $L^{2}(\varOmega , \mathcal{F}, \mathbh{P})$ generated by the
family of random variables $\{H_{n}(W(h)):~\|h\|_{\mathbh{H}}=1 \}$. The
vector spaces $\mathcal{H}_{n}$, $n\geq 0$, are orthogonal and we have
the Wiener chaos expansion\index{Wiener chaos expansion}
%
\begin{align}
\label{Wienerchaos} L^{2}(\varOmega , \mathcal{F}, \mathbh{P})=
\bigoplus_{n=0}^{\infty } \mathcal{H}_{n}.
\end{align}
(See \cite[Theorem 1.1.1]{N} for more details\xch{.)}{).} Next, for $\alpha
\in \mathcal{I}$, where $\mathcal{I}$ has been introduced in
\eqref{setI}, we define the random variable
%
\begin{align}
\label{psialpha} \varPsi ^{\alpha } := \prod_{i=1}^{\infty }H_{\alpha _{i}}
\bigl(W(e_{i})\bigr),
\end{align}
where $(e_{i})_{i \geq 1}$ is a fixed orthonormal basis\index{orthonormal basis} of $
\mathbh{H}$. We define $|\alpha |= \sum_{i=1}^{\infty } \alpha _{i}$ and
$\alpha !=\prod_{i=1}^{\infty } \alpha _{i}!$ for $\alpha \in
\mathcal{I}$. We deduce that the set $\{ \varPsi ^{\alpha }:~ \alpha
\in \mathcal{I} \text{ with } |\alpha |=n \}$ forms a complete
orthonormal basis\index{orthonormal basis} of the $n$th Wiener chaos\index{Wiener chaos} $\mathcal{H}_{n}$ and
consequently $\{ \varPsi ^{\alpha }:~ \alpha \in \mathcal{I} \}$ is a
complete orthonormal basis\index{orthonormal basis} of $L^{2}(\varOmega , \mathcal{F}, \mathbh{P})$
(cf. \cite[Proposition 1.1.1]{N}).

Now, we introduce multiple stochastic integrals of order $n \in
\mathbh{N}$, which are denoted by $I_{n}$. For an element $h^{\otimes
n}:=h \otimes \cdots \otimes h$ in $\mathbh{H}^{\otimes n}$ with
$\|h\|_{\mathbh{H}}=1$, we define
%
\begin{align}
\label{defI} I_{n}\bigl(h^{\otimes n}\bigr):= \sqrt{n!}
~H_{n} \bigl(W(h) \bigr), \quad n\geq 1.
\end{align}
Assuming that the mapping $I_{n}$ is linear, the definition
\eqref{defI} can be extended to all symmetric elements $h \in
\mathbh{H}^{\otimes n}$ by polarisation identity.\vadjust{\goodbreak} Finally, for an
arbitrary function $h \in \mathbh{H}^{\otimes n}$ we set
\begin{equation*}
I_{n} (h):= I_{n} (\widetilde{h}),
\end{equation*}
where $\widetilde{h}$ denotes the symmetrised version of $h \in
\mathbh{H}^{\otimes n}$. By definition $I_{n}$ maps $\mathbh{H}$ into
$\mathcal{H}_{n}$, so the multiple integrals of different orders are
orthogonal. In particular, they satisfy the isometry property
%
\begin{align}
\label{isometry} \mathbh{E} \bigl[ I_{m}(g) I_{n}(h)
\bigr] = n! \langle \widetilde{g}, \widetilde{h}\rangle _{\mathbh{H}^{\otimes n}}
1_{\{n=m
\}}, \quad h \in \mathbh{H}^{\otimes n}, g \in
\mathbh{H}^{\otimes m}.
\end{align}
Furthermore, for any symmetric $h \in \mathbh{H}^{\otimes n}, g
\in \mathbh{H}^{\otimes m}$, the following multiplication formula holds:
%
\begin{align}
\label{multiplication} I_{m}(g) I_{n}(h) = \sum
_{r=0}^{\min (m,n)} r! {m \choose r}
{n \choose r} I_{m+n-2r} (g \otimes _{r} h ),
\end{align}
where the $r$th \textit{contraction} $g \otimes _{r} h$ is defined by
\begin{align*}
g \otimes _{r} h (t_{1}, \ldots , t_{m+n-2r})
&:= \int_{[0,T]^{r}} g(u _{1},\ldots
,u_{r},t_{1},\ldots t_{m-r})
\\[1.5 ex] &\qquad  \times h(u_{1},\ldots ,u_{r},t_{m-r+1},
\ldots t_{m+n-2r}) du_{1} \ldots du_{r}.
\end{align*}
(See \cite[Proposition 1.1.3]{N}.) The Wiener chaos expansion\index{Wiener chaos expansion}
\eqref{Wienerchaos} transfers to the context of multiple integrals as
follows. For each random variable $F \in L^{2}(\varOmega , \mathcal{F},
\mathbh{P})$ we obtain the orthogonal decomposition
%
\begin{align}
\label{Wienerchaos2} F = \sum_{n=0}^{\infty }
I_{n} (g_{n}), \quad g_{n} \in
\xch{\mathbh{H}^{\otimes n},}{\mathbh{H}^{\otimes n}}
\end{align}
where $I_{0} = \mathbh{E}[F]$ and the decomposition is unique when
$g_{n}$, $n\geq 1$, are assumed to be symmetric (cf.
\cite[Theorem 1.1.2]{N}).

Next, we introduce the notion of Malliavin derivative\index{Malliavin derivative} and its adjoint
operator. We define the set of smooth random variables via
\begin{equation*}
\mathcal{S} = \bigl\{ F = f \bigl( W(h_{1}), \ldots ,
W(h_{n}) \bigr):~ n \geq 1, h_{i} \in \mathbh{H} \bigr
\},
\end{equation*}
where $f \in C_{p}^{ \infty } (\mathbh{R}^{n})$ (i.e. the space of
infinitely
differentiable functions such that all derivatives
exhibit polynomial growth). The $k$\textit{th order Malliavin
derivative\index{Malliavin derivative}} of $F \in \mathcal{S}$, denoted by $D^{k} F$, is defined by
%
\begin{align}
\label{Derivative} D^{k} F = \sum_{i_{1}, \ldots , i_{k} = 1}^{n}
\frac{ \partial ^{k}}{
\partial x_{i_{1}} \cdots \partial x_{i_{k}}} f \bigl(W(h_{1}), \ldots , W(h_{n})
\bigr) h_{i_{1}} \otimes \cdots \otimes h_{i_{k}}.
\end{align}
Notice that $D^{k} F$ is a $\mathbh{H}^{\otimes k}$-valued random
variable, and we write $D_{x}^{k} F$ for the realisation of the function
$D^{k} F$ at the point $x \in [0,T]^{k}$. The space $\mathbh{D}^{k,q}$
denotes the completion of the set $\mathcal{S}$ with respect to the norm
\begin{equation*}
\llVert F \rrVert _{k, q} := \Biggl( \mathbh{E}\bigl[ \llvert F
\rrvert ^{q} \bigr] + \sum_{m =
1}^{k}
\mathbh{E}\bigl[ \bigl\llVert D^{m} F \bigr\rrVert
_{ \mathbh{H}^{ \otimes m}}^{q}\bigr] \Biggr) ^{1/q}.
\end{equation*}
We define $\mathbh{D}^{k,\infty }= \cap _{q>1} \mathbh{D}^{k,q}$. The
Malliavin derivative\index{Malliavin derivative} of the random variable $\varPsi ^{\alpha } \in
\mathcal{S}$, $\alpha \in \mathcal{I}$, can be easily computed\vadjust{\goodbreak} using the
definition \eqref{Derivative} and the formula $H'_{n}(x)= \sqrt{n}H
_{n-1}(x)$:
%
\begin{align}
\label{derPsi} D_{s} \varPsi ^{\alpha } &= \sum
_{i=1}^{\infty } \sqrt{\alpha _{i}} e
_{i}(s) \xch{\varPsi ^{\alpha ^{-}(i)},}{\varPsi ^{\alpha ^{-}(i)}} \quad \text{where}
\\
\label{alphaminus} \alpha ^{-}(i) &:= (\alpha _{1},
\ldots , \alpha _{i-1}, \alpha _{i} -1, \alpha
_{i+1}, \ldots ) \quad \text{if } \alpha _{i} \geq 1.
\end{align}
Higher order Malliavin derivatives\index{Malliavin derivative} of $\varPsi ^{\alpha }$ are computed
similarly.

The operator $D^{k}$ possesses an unbounded adjoint denoted by
$\delta ^{k}$, which is often referred to as multiple Skorokhod integral.
The following integration by parts formula holds (see
\cite[Exercise 1.3.7]{N}): if $u \in \text{Dom}( \delta ^{k})$ and
$F \in \mathbh{D}^{k,2}$, where $\text{Dom}( \delta ^{k})$ consists of
all elements $u \in L^{2}(\varOmega ; \mathbh{H}^{ \otimes k})$ such that
the inequality $|\mathbh{E}[ \langle D^{k} F, u
\rangle _{ \mathbh{H}^{ \otimes k}}]| \leq c (E[F^{2}])^{1/2}$ holds for
some $c>0$ and all $F \in \mathbh{D}^{k,2}$, then we have the identity
%
\begin{align}
\label{IBP} \mathbh{E} \bigl[ F \delta ^{k}(u) \bigr] =
\mathbh{E} \bigl[ \bigl\langle D ^{k} F, u \bigr\rangle
_{ \mathbh{H}^{ \otimes k}} \bigr].
\end{align}
In case the random variable $F \,{\in}\, L^{2}(\varOmega , \mathcal{F},
\mathbh{P})$ has a chaos decomposition as \mbox{displayed} in
\eqref{Wienerchaos2}, the statement $F \in \mathbh{D}^{k,2}$ is
equivalent to the condition $\sum_{n=1}^{\infty } n^{k} n! \|g_{n}\|
^{2}_{\mathbh{H}^{\otimes n}} < \infty $. In particular, when
$F \in \mathbh{D}^{1,2}$ we deduce an explicit chaos representation of
the derivative $DF=(D_{t} F)_{t \in [0,T]}$
%
\begin{align}
\label{Dt} D_{t} F = \sum_{n=1}^{\infty }
n I_{n-1} \bigl(g_{n} (\cdot , t) \bigr), \quad t \in
[0,T],
\end{align}
where $g_{n} (\cdot , t): [0,T]^{n-1}\to \mathbh{R}$ is obtained from
the function $g_{n}$ by setting the last argument equal to $t$ (see
\cite[Proposition 1.2.7 ]{N}). Finally, we present an explicit formula
for the Malliavin derivative\index{Malliavin derivative} of a solution 
of a stochastic differential
equation. Assume that $(X_{t})_{t \in [0,T]}$ is a solution of a
stochastic differential equation \eqref{sde} and $b,\sigma \in C^{1}(
\mathbh{R})$. In this setting $DX_{t}=(D_{s}X_{t})_{s \in [0,T]}$ is
given as the solution of the SDE
%
\begin{align}
\label{MalSDE} D_{s}X_{t} = \sigma (s,
X_{s}) + \int_{s}^{t}
\frac{\partial }{\partial
x}b(u,X_{u})D_{s}(X_{u})
du + \int_{s}^{t} \frac{\partial }{\partial
x} \sigma
(u,X_{u}) D_{s}(X_{u})
dW_{u},
\end{align}
for $s \leq t$, and $D_{s}X_{t} = 0$ if $s>t$ (see
\cite[Theorem 2.2.1]{N}). Throughout the paper $\partial /\partial x$
denotes the derivative with respect to the space variable and
$\partial /\partial t$ denotes the derivative with respect to the time
variable.

\section{Main results}%
\label{sec3}
We start with the analysis of the Wiener chaos expansion\index{Wiener chaos expansion} introduced 
in \eqref{choasx}. Under square integrability assumption on the solution
$X$ of the SDE \eqref{sde} we obtain the Wiener chaos expansion\index{Wiener chaos expansion}
\begin{align*}
X_{t} = \sum_{\alpha \in \mathcal{I}} x_{\alpha }
(t) \varPsi ^{\alpha } \quad \text{with} \ \varPsi ^{\alpha }=
\prod_{i=1}^{\infty }H_{\alpha _{i}}
\bigl(W(e_{i})\bigr).
\end{align*}
In order to study the strong approximation error we require a good
control of the coefficients $x_{\alpha } (t)$.\vadjust{\goodbreak} When $X_{t}$ is
sufficiently smooth in the Malliavin sense, we deduce the identity
%
\begin{align}
\label{xformula} x_{\alpha } (t) = \mathbh{E} \bigl[ X_{t}
\varPsi ^{\alpha } \bigr] = \frac{1}{\sqrt{
\alpha !}} \mathbh{E} \Biggl[ \Biggl
\langle D^{|\alpha |} X_{t}, \bigotimes
_{i = 1}^{\infty } e_{i}^{\otimes \alpha _{i}}
\Biggr\rangle _{\mathbh{H}^{\otimes |\alpha |}} \Biggr]
\end{align}
applying the duality formula \eqref{IBP}. In fact, the coefficients
$x_{\alpha } (t)$ satisfy a system of ordinary differential equations.
The following propagator\index{propagator system} system has been derived in
\cite{HS14,LMR97}. We state the proof for completeness.

\begin{theo}
\label{th1}
Let $(X_{t})_{t \in [0,T]}$ be the unique solution of the SDE
\eqref{sde} and assume that $X \in L^{2}(\varOmega \times [0,T])$. Then
$X_{t}$ possesses the chaos expansion\index{chaos expansion} 
\eqref{choasx} and the
coefficients $x_{\alpha } (t)$ satisfy the system of ordinary
differential equations
%
\begin{align}
\label{propagator} x'_{\alpha }(t) &= b_{\alpha }
(t,X_{t}) + \sum_{j=1}^{\infty }
\sqrt{ \alpha _{j}} e_{j}(t) \sigma _{\alpha ^{-}(j)}
(t,X_{t}),
\\
x_{\alpha }(0) &= 1_{\{\alpha =0\}} x_{0}.
\nonumber
\end{align}
Here $b_{\alpha } (t,X_{t})$ (resp. $\sigma _{\alpha } (t,X_{t})$)
denotes the $\alpha $-coefficient of the Wiener chaos expansion
\eqref{choasx} associated with the random variable $b(t, X_{t})$ (resp.
$\sigma (t,X_{t})$), and the multi-index $\alpha ^{-}(j)$ is defined 
by \eqref{alphaminus}.
\end{theo}

\begin{proof}
Using the SDE \eqref{sde} and applying the formula \eqref{xformula} we
obtain the identity
\begin{align*}
x_{\alpha }(t) = x_{0}1_{\{ \alpha = 0\}} + \int
_{0}^{t} \mathbh{E}\bigl[ \varPsi
^{\alpha } b(s,X_{s})\bigr] ds+ \mathbh{E} \Biggl[\varPsi
^{\alpha } \int_{0}^{t} \sigma
(s,X_{s}) dW_{s} \Biggr].
\end{align*}
Applying the formula \eqref{xformula} once again for the random variable
$b(s,X_{s})$ we immediately deduce that
\begin{equation*}
b_{\alpha } (s,X_{s})= \mathbh{E}\bigl[\varPsi
^{\alpha } b(s,X_{s})\bigr].
\end{equation*}
On the other hand, observing the identity $\delta (1_{[0,t]}\sigma (
\cdot , X_{\cdot })) = \int_{0}^{t} \sigma (s,X_{s}) dW_{s}$, we get by
the duality formula \eqref{IBP} and \eqref{derPsi} that
\begin{align*}
\mathbh{E} \Biggl[\varPsi ^{\alpha } \int_{0}^{t}
\sigma (s,X_{s}) dW_{s} \Biggr] &= \sum
_{j=1}^{\infty } \int_{0}^{t}
\sqrt{\alpha _{j}} e _{j}(s) \mathbh{E} \bigl[\varPsi
^{\alpha ^{-}(j)} \sigma (s,X_{s}) \bigr] ds
\\
&= \sum_{j=1}^{\infty } \int
_{0}^{t} \sqrt{\alpha _{j}}
e_{j}(s) \sigma _{\alpha ^{-}(j)}(s,X_{s}) ds.
\end{align*}
Putting things together we obtain the identity
\begin{equation*}
x_{\alpha }(t) = x_{0}1_{\{ \alpha = 0\}} + \int
_{0}^{t} b_{\alpha } (s,X
_{s}) ds + \sum_{j=1}^{\infty }
\int_{0}^{t} \sqrt{\alpha _{j}} e
_{j}(s) \sigma _{\alpha ^{-}(j)}(s,X_{s}) ds.
\end{equation*}
Consequently, the assertion follows after taking the derivative with
respect to $t$.
\end{proof}
We remark that the propagator\index{propagator system} system \eqref{propagator} is recursive.
Let us give some simple examples to illustrate how \eqref{propagator}
can be solved explicitly.

\begin{ex}
\label{ex32}
(i) (\textit{Scaled Brownian motion with drift}) We start with the toy
example of a scaled Brownian motion with drift:
\begin{equation*}
X_{t} = bt +\sigma W_{t}.\vadjust{\goodbreak}
\end{equation*}
In this case we obviously have that $x_{\alpha }(t)=0$ for any
$\alpha \in \mathcal{I}$ with $|\alpha |\geq 2$. Applying formula
\eqref{propagator} we obtain the representation
\begin{equation*}
X_{t} = bt + \sigma \sum_{j=1}^{\infty }
\Biggl( \int_{0}^{t} e_{j}(s) ds
\Biggr) \int_{0}^{T} e_{j}(s)
dW_{s},
\end{equation*}
which is a well known Karhunen--Lo\'{e}ve expansion of the scaled
Brownian motion.

(ii) (\textit{Geometric Brownian motion}) Let us consider the geometric
Brownian motion defined via the SDE
\begin{equation*}
dX_{t} = bX_{t} dt + \sigma X_{t}
dW_{t} , \quad X_{0}=x_{0}>0.
\end{equation*}
In this setting the propagator\index{propagator system} system \eqref{propagator} translates to
\begin{align*}
x'_{\alpha }(t) &= b x_{\alpha }(t) + \sigma \sum
_{j=1}^{\infty }\sqrt{ \alpha
_{j}} e_{j}(t) x_{\alpha ^{-}(j)}(t),
\\
x_{\alpha }(0) &= 1_{\{\alpha =0\}} x_{0}.
\end{align*}
This system of ordinary differential equations can be solved
recursively. For $\alpha =0$ we have $x'_{0}(t) = b x_{0}(t)$ and hence
$x_{0}(t) = x_{0} \exp (bx)$. If $\alpha $ is the $j$th canonical unit
vector in $\mathcal{I}$ (and hence $|\alpha |=1$) we obtain the
differential equation
\begin{align*}
x'_{\alpha }(t) &= b x_{\alpha }(t) + \sigma
e_{j}(t) x_{0}(t),
\\
x_{\alpha }(0) &= 1_{\{\alpha =0\}} x_{0}.
\end{align*}
Hence, $x_{\alpha }(t)= x_{0} \sigma \exp (bx) \int_{0}^{t} e_{j}(s)
ds$. Following this recursion we obtain the general formula
\begin{equation*}
x_{\alpha }(t) = \frac{1}{\sqrt{\alpha !}} x_{0} \sigma
^{p} \exp (bx) \prod_{j=1}^{\infty }
\Biggl(\int_{0}^{t} e_{j}(s) ds
\Biggr)^{\alpha
_{j}}
\end{equation*}
for any $\alpha \in \mathcal{I}$ with $|\alpha |=p$.
\end{ex}
For a general specification of the drift coefficient $b$ and diffusion
coefficient $\sigma $ in model \eqref{sde} the propagator\index{propagator system} system
\eqref{propagator} cannot be solved explicitly. Thus, precise infinite
dimensional Wiener chaos expansion\index{Wiener chaos expansion} \eqref{choasx} is out of reach. For
simulation purposes it is an obvious idea to consider a finite subset
of $\mathcal{I}$ in the expansion \eqref{choasx}. We introduce the index
set
%
\begin{align}
\label{ipk} \mathcal{I}_{p,k}:= \bigl\{\alpha \in \mathcal{I}:~
\llvert \alpha \rrvert \leq p\  \text{ and }\  \alpha _{i}=0 \text{ for
all } i>k \bigr\}.
\end{align}
The approximation of $X_{t}$ is now defined via \eqref{xpk}:
%
\begin{align}
\label{xpkt} X_{t}^{p,k} = \sum
_{\alpha \in \mathcal{I}_{p,k}} x_{\alpha } (t) \varPsi ^{\alpha }.
\end{align}
We remark that the quantity $X_{t}^{p,k}$ is more useful than the Euler
approximation $X_{t}^{n}$ introduced 
in \eqref{Eulersde} if we are
interested in the approximation of the first two moments of
$X_{t}$. Indeed, the first two moments of $X_{t}^{p,k}$ are given
explicitly by
\begin{equation*}
\mathbh{E} \bigl[X_{t}^{p,k} \bigr] = x_{0}
(t) \quad \text{and} \quad \mathbh{E} \bigl[\bigl(X_{t}^{p,k}
\bigr)^{2} \bigr] = \sum_{\alpha \in \mathcal{I}_{p,k}}
x_{\alpha }^{2} (t),
\end{equation*}
while higher order moments can be computed via an application of the
multiplication formula \eqref{multiplication}.

The strong approximation error associated with the truncation 
\eqref{xpkt} has been previously studied in \cite{LMR97,L06} in a
slightly different context. In particular, in both papers the authors
consider one specific basis $(e_{i})_{i\geq 1}$ of $L^{2}([0,T])$
whereas we are interested in the asymptotic analysis for a general basis
$(e_{i})_{i\geq 1}$. While \cite{L06} mostly uses methods from
analysis, our approach is based upon Malliavin calculus\index{Malliavin calculus} and 
is close
in spirit to \cite{LMR97}. The main result of our paper gives an upper
bound on the $L^{2}$-error $\mathbh{E}[(X_{t}^{p,k} - X_{t})^{2}]$.

\begin{theo}
\label{th2}
Let $(X_{t})_{t \in [0,T]}$ be the solution of the SDE \eqref{sde}.
Suppose that the coefficient functions $b$ and $\sigma $ satisfy the
Lipschitz and linear growth conditions\index{linear growth conditions}
%
\begin{align}
\label{growth} \bigl\llvert b(t,x) - b(t,y) \bigr\rrvert + \bigl\llvert \sigma
(t,x) - \sigma (t,y) \bigr\rrvert &\leq K \llvert x-y \rrvert , \quad t\in
[0,T],
\\[1.5 ex] \bigl\llvert b(t,x) \bigr\rrvert ^{2} + \bigl\llvert \sigma
(t,x) \bigr\rrvert ^{2} &\leq K^{2} \bigl(1+ \llvert x
\rrvert ^{2} \bigr), \quad t\in [0,T].
\nonumber
\end{align}
Moreover, assume that $b, \sigma \in C^{1,\infty }([0,T] \times
\mathbh{R})$, where $C^{1,\infty }([0,T] \times \mathbh{R})$ denotes the
space of functions $f:[0,T] \times \mathbh{R}\to \mathbh{R}$ that are
once continuously differentiable in the first component and infinitely
differentiable in the second component, such that
\begin{align*}
\biggl\llvert \frac{\partial ^{l+m}}{\partial t^{l} \partial x^{m}} b(t,x) -
\frac{\partial ^{l+m}}{\partial t^{l} \partial x^{m}}b(t,y) \biggr
\rrvert + \biggl\llvert \frac{\partial ^{l+m}}{\partial t^{l} \partial x^{m}} \sigma (t,x) -
\frac{\partial ^{l+m}}{\partial t^{l} \partial x^{m}}
\sigma (t,y) \biggr\rrvert &\,{\leq}\, K \llvert x\,{-}\,y \rrvert ,
\end{align*}
for $t\in [0,T]$, any $l=0,1$, and $m\geq 0$. Then it holds that
%
\begin{align}
\label{estimate} \mathbh{E}\bigl[\bigl(X_{t}^{p,k} -
X_{t}\bigr)^{2}\bigr] \leq C \bigl(1+{x_{0}}^{2}
\bigr) \Biggl( \frac{1}{(p+1)!} + \sum_{l=k+1}^{\infty }
\Biggl( E_{l}^{2}(t) + \int_{0}^{t}
E_{l}^{2}(\tau ) d\tau \Biggr) \Biggr),
\end{align}
where $C=C(t,K)$ is a positive constant and the function $E_{l}(t)$ is
defined by
%
\begin{align}
\label{defE} E_{l}(t):= \int_{0}^{t}
e_{l}(s) ds.
\end{align}
\end{theo}
Let us give some remarks about the statement \eqref{estimate}. First of
all, recalling the Karhunen--Lo\'{e}ve expansion $W_{t} = \sum_{l=1}
^{\infty } E_{l}(t) \int_{0}^{T} e_{l}(s) dW_{s}$, we readily deduce
that
\begin{align*}
\mathbh{E}\bigl[W_{t}^{2}\bigr] &= \sum
_{l=1}^{\infty } E_{l}^{2} (t)
< \infty \quad \text{and}
\\[1.5 ex] \int_{0}^{t} \mathbh{E}
\bigl[W_{s}^{2}\bigr] ds &= \sum
_{l=1}^{\infty } \int_{0}^{t}
E_{l}^{2}(\tau ) d\tau < \infty .
\end{align*}
Hence, the upper bound on the right-hand side of \eqref{estimate} indeed
converges to $0$ when $p,k\to \infty $. We also note that the error
associated with the truncation of the chaos order does not depend on the
particular choice of the basis $(e_{i})_{i\geq 1}$, while the error
associated with truncation of basis strongly depends on $(e_{i})_{i
\geq 1}$ (which is not really surprising). Furthermore,\vadjust{\goodbreak} we remark that
it is extremely computationally costly to compute all coefficients
$x_{\alpha }(t)$, $\alpha \in \mathcal{I}_{p,k}$, for a large chaos
order $p$. Thanks to the first bound $((p+1)!)^{-1}$ 
in \eqref{estimate} it is sufficient to use small values of $p$ in
practical situations (usually $p\leq 4$).

\begin{ex}
\label{ex}
Let us explicitly compute the last terms of the upper bound
\eqref{estimate} for two prominent bases of $L^{2}([0,1])$.

(i) (\textit{Trigonometric basis\index{trigonometric basis}}) Consider the orthonormal basis\index{orthonormal basis}
$(e_{i})_{i\geq 1}$ given by
\begin{equation*}
e_{1}(t)=1, \qquad e_{2j}(t)= \sqrt{2} \sin (2\pi jt ),
\qquad e_{2j+1}(t)= \sqrt{2} \cos (2\pi jt ), \quad j \geq 1.
\end{equation*}
In this setting we obtain that
\begin{align*}
E_{2j}^{2} (t) = \frac{1}{2\pi ^{2} j^{2}} \bigl(1-\cos (2
\pi jt) \bigr) ^{2}, \qquad E_{2j+1}^{2} (t) =
\frac{1}{2\pi ^{2} j^{2}} \bigl(1-\sin (2\pi jt) \bigr) ^{2},
\end{align*}
for any $j\geq 1$. Consequently, we deduce that
%
\begin{align}
\sum_{l=k+1}^{\infty } \Biggl(
E_{l}^{2}(t) + \int_{0}^{t}
E_{l}^{2}( \tau ) d\tau \Biggr) \leq
\frac{C}{k} \label{eq_rate_Trigonometric}
\end{align}
for some $C>0$.

(ii) (\textit{Haar basis\index{Haar basis}}) The Haar basis\index{Haar basis} is a collection of functions
\begin{equation*}
\bigl\{e_{0}, e_{j,n}:~j=1,\ldots , 2^{n-1},
n\geq 1 \bigr\}
\end{equation*}
defined as follows:
\begin{align*}
e_{0}(t)=1, \qquad e_{j,n}(t) = %
\begin{cases}
2^{(n-1)/2}:
& t \in [2^{-n+1}(j-1), 2^{-n}(2j-1)\xch{],}{]}
\\
-2^{(n-1)/2}:
& t \in [2^{-n}(2j-1), 2^{-n+1}j\xch{],}{]}
\\
0:
& \xch{\text{else}.}{\text{else}}
\end{cases}
\end{align*}
In this case we deduce the following representation for $E_{j,n}(t)$:
\begin{align*}
E_{j,n}(t) = %
\begin{cases}
2^{(n-1)/2} (t-2^{-n+1}(j-1)):
& t \in [2^{-n+1}(j-1), 2^{-n}(2j-1)\xch{],}{]}
\\
2^{(n-1)/2} (2^{-n+1}j-t):
& t \in [2^{-n}(2j-1), 2^{-n+1}j\xch{],}{]}
\\
0:
& \xch{\text{else}.}{\text{else}}
\end{cases} %
\end{align*}
Since $\max_{t \in [0,1]}E_{j,n}^{2}(t) = 2^{-(n+1)}$ and $E_{j,n}(t)
\neq 0$ only for $t \in [2^{-n+1}(j-1),2^{-n+1}j]$, we finally obtain
that
%
\begin{align}
\sum_{l=n+1}^{\infty } \sum
_{j=1}^{2^{n}} \Biggl( E_{j,l}^{2}(t)
+ \int_{0}^{t} E_{j,l}^{2}(
\tau ) d\tau \Biggr) \leq C 2^{-n}. \label{eq_rate_Haar}
\end{align}
Note that the exponential asymptotic rate for a fixed $p$th order Wiener
chaos\index{Wiener chaos} is equivalent to $\mathcal{O}(k^{-1})$, if we choose $k=2^{n}$
basis elements $(e_{i})$. We will come back to these two bases in
Section~\ref{sec5} where we provide exemplary numerical results that
highlight this identical asymptotic rate, but also the different error
distributions over time.
\end{ex}

\section{Proofs}%
\label{sec4}
Throughout the proofs $C=C(t,K)$ denotes a generic positive constant,
which might change from line to line. We start with a proposition that
gives an upper bound for the $L^{2}$-norm of the Malliavin derivatives\index{Malliavin derivative}
of $X_{t}$.\vadjust{\goodbreak}

\begin{prop}
\label{prop1}
Under conditions of Theorem \ref{th2} we obtain the estimate
%
\begin{align}
\label{Malestimate} \mathbh{E} \bigl[ \bigl( D^{n}_{s_{1}, \ldots , s_{n}}
X_{t} \bigr) ^{2} \bigr] \leq C^{n}
\bigl(1+x_{0}^{2}\bigr).
\end{align}
\end{prop}

\begin{proof}
We show the assertion of Proposition \ref{prop1} by induction over
$n$. For $n=0$ the result is a well-known consequence of the Lipschitz
and linear growth conditions\index{linear growth conditions} 
\eqref{growth} on the functions $b$ and
$\sigma $ (see e.g. \cite[Theorem 2.9, page 289]{KS} which relates the
second moment of $X_{t}$ to the second moment of the initial condition
$X_{0}=x_{0}$). For $n=1$ we known from the formula \eqref{MalSDE} that
the process $(D_{s}X_{t})_{s\in [0,T]}$ satisfies the SDE
\begin{align*}
D_{s}X_{t} = \sigma (s, X_{s}) + \int
_{s}^{t} \frac{\partial }{\partial
x}
b(u,X_{u})D_{s}(X_{u}) du + \int
_{s}^{t} \frac{\partial }{\partial
x} \sigma
(u,X_{u}) D_{s}(X_{u})
dW_{u},
\end{align*}
for $s\leq t$, and $D_{s}X_{t} =0$ for $s>t$. Next, we introduce the
two-dimensional stochastic process $(Y^{(1)}_{s;t})_{t \geq s}$ via
\begin{equation*}
Y^{(1)}_{s;t} := (X_{t},
D_{s}X_{t}).
\end{equation*}
Note that the process $(Y^{(1)}_{s;t})_{t \geq s}$ satisfies a
two-dimensional SDE with initial condition given by $Y^{(1)}_{s;s}$. In
particular, setting $Y^{(1)}_{s;t} =   (Y^{(1.1)}_{s;t},Y^{(1.2)}
_{s;t}   )$, we have the representation
\begin{equation*}
D_{s}X_{t} = \sigma \bigl(s, Y^{(1.1)}_{s;s}
\bigr) + \int_{s}^{t} \overline{b}
\bigl(u,Y^{(1)}_{s;u} \bigr) du + \int_{s}^{t}
\overline{ \sigma } \bigl(u,Y^{(1)}_{s;u} \bigr)
dW_{u},
\end{equation*}
where $\overline{b} (u,y)= \frac{\partial }{\partial x} b(u,y_{1}) y
_{2}$ and $\overline{\sigma } (u,y)= \frac{\partial }{\partial x}
\sigma (u,y_{1}) y_{2}$ for $y=(y_{1},y_{2})$. Applying again the
estimates of \cite[Theorem 2.9, page 289]{KS} to the diffusion process
$(Y^{(1)}_{s;t})_{t \geq s}$, we conclude the estimate
%
\begin{align}\nonumber
\mathbh{E} \bigl[ ( D_{s} X_{t} )^{2}
\bigr] &\leq \mathbh{E} \bigl[ \bigl\llVert Y^{(1)}_{s,t}
\bigr\rrVert ^{2} \bigr] \leq C \bigl(1+\mathbh{E} \bigl[ \bigl\llVert
Y^{(1)}_{s,s} \bigr\rrVert ^{2} \bigr] \bigr)
\\ &\leq C \bigl(1+\mathbh{E} \bigl[ X_{s} ^{2}
\bigr] + \mathbh{E} \bigl[ \sigma (s,X_{s})^{2} \bigr]
\bigr)
\nonumber
\\[1.5 ex] &\leq C\bigl(1+x_{0}^{2}\bigr) .
\end{align}
Now, we will perform the induction step. Notice that $D_{s}X_{t}
\in \mathbh{D}^{1,\infty }$, so its Malliavin derivative\index{Malliavin derivative} satisfies again
an SDE according to \cite[Theorem 2.2.2]{N}. We define the stochastic
process
\begin{equation*}
Y^{(n)}_{s_{1}, \ldots ,s_{n} ;t} := \bigl(X_{t},
D_{s_{1}}X_{t}, D ^{2}_{s_{1}, s_{2}}X_{t},
\ldots , D^{n}_{s_{1}, \ldots , s_{n}}X_{t} \bigr).
\end{equation*}
Assume that the assertion of Theorem \ref{th2} holds for all
$k=1,\ldots ,n-1$. In order to compute the estimate, we have to consider
the initial values of the SDE system satisfied by $Y_{s_{1},\ldots ,s
_{n};\,t}^{(n)}$. The formula below can be found in
\cite[Theorem 2.2.2]{N}. Before we state it, we need some additional
notation. The stochastic process $D^{m}X_{t} =   \{D^{m}_{s_{1},
\ldots ,s_{m}}X_{t} \bigm | (s_{1},\ldots ,s_{m})\in [0,T]  \}$
depends on the $m$ time points $s_{1},\ldots ,s_{m}$. For any subset
$J=\{j_{1}<\cdots <j_{\eta }\}$ of $\{1,\ldots ,m\}$ with $|J|=\eta
\leq m$ elements, denote $s(J) = (s_{j_{1}},\ldots ,s_{j_{\eta }})$.
Further on, we define
%
\begin{align}
\label{frakz} \mathfrak{z}(t,s_{1},\ldots ,s_{m}) =
\sum_{\mathfrak{P}^{m}} \frac{
\partial ^{m}}{\partial x^{m}} \sigma
(t,X_{t}) D^{|s(J_{1})|}_{s(J
_{1})}X_{t}
\cdots D^{|s(J_{\nu })|}_{s(J_{\nu })}X_{t},
\end{align}
and
%
\begin{align}
\label{fraky} \mathfrak{y}(t,s_{1},\ldots ,s_{m}) =
\sum_{\mathfrak{P}^{m}} \frac{
\partial ^{m}}{\partial x^{m}}
b(t,X_{t}) D^{|s(J_{1})|}_{s(J_{1})}X _{t}
\cdots D^{|s(J_{\nu })|}_{s(J_{\nu })}X_{t},
\end{align}
where the sums run over the set $\mathfrak{P}^{m}$ of all partitions
$J_{1}\cup \cdots \cup J_{\nu }$ of $\{1,\ldots ,m\}$, where
$J_{1}, \ldots , J_{\nu }$ are disjoint sets. We determine $
\mathfrak{z}(t) = \sigma (t,X_{t})$ as well. With these notations at
hand, we find by \eqref{MalSDE} and induction that the $n$th order
Malliavin derivative $D^{n}_{s_{1},\ldots ,s_{n}}X_{t}$ satisfies the
linear SDE
%
\begin{align}
\label{Dsn} D^{n}_{s_{1},\ldots ,s_{n}}X_{t} &= \sum
_{i=1}^{n} \mathfrak{z}(s_{i},s
_{1},\ldots ,s_{i-1},s_{i+1},\ldots
,s_{n})
\nonumber
\\
& \qquad + \int_{\hat{s}}^{t}
\mathfrak{y}(u,s_{1},\ldots ,s_{n}) du + \int
_{\hat{s}}^{t}\mathfrak{z}(u,s_{1},
\ldots ,s_{n})dW_{u}
\end{align}
for $\hat{s}:=\max \{s_{1},\ldots ,s_{n}\}\leq t$ and $D^{n}_{s_{1},
\ldots ,s_{n}}X_{t}=0$ else. Hence, its initial value is given by
\begin{align*}
\sum_{i=1}^{n} \mathfrak{z}(s_{i},
s_{1}, \ldots , s_{i-1}, s_{i+1}, \ldots ,
s_{n}),
\end{align*}
where
\begin{align*}
&\mathfrak{z}(s_{1},\ldots ,s_{n}) =
\frac{\partial ^{n}}{\partial x
^{n}}\sigma (s_{1},X_{s_{1}}\xch{)}{) \times}
\\
& \qquad \times \bigl(D_{s_{2}}X_{s_{1}}\cdots
D_{s_{n}}X_{s_{1}} + D ^{2}_{s_{2},s_{3}}X_{s_{1}}
\cdot D_{s_{4}}X_{s_{1}}\cdots D_{s_{n}}X
_{s_{1}} + \cdots + D^{n-1}_{s_{2},\ldots ,s_{n}}X_{s_{1}}
\bigr).
\end{align*}
Finally, we apply \cite[Theorem 2.9, page 289]{KS} as for the case
$n=1$ and taking into account that the assertion holds for $k=1,
\ldots , n-1$:
\begin{align*}
&\mathbh{E} \bigl[ \bigl\llvert D^{n}_{s_{1},\ldots ,s_{n}}X_{t}
\bigr\rrvert ^{2} \bigr]\leq \mathbh{E} \bigl[ \bigl\llVert
Y_{s_{1},\ldots ,s_{n};\,t}^{(n)} \bigr\rrVert ^{2} \bigr]
\\ & \qquad \leq C \Biggl( 1 + \mathbh{E} \bigl[ \llvert
X_{\hat{s}} \rrvert ^{2} \bigr] + \cdots + \mathbh{E} \Biggl[
\Biggl\llvert \sum_{i=1}^{k}
\mathfrak{z}(s_{i},s _{1},\ldots s_{i-1},s_{i+1},
\ldots ,s_{n}) \Biggr\rrvert ^{2} \Biggr] \Biggr)
\\ & \qquad \leq C^{n} \bigl(1+{x_{0}}^{2}
\bigr) .
\end{align*}
This completes the proof of Proposition \ref{prop1}.
\end{proof}

Now, we proceed with the proof of the main result of Theorem
\ref{th2}, which follows the ideas of
\cite[Proof of Theorem 2.2]{LMR97}. Observe the decomposition
%
\begin{align}
\label{decomposition} \mathbh{E} \bigl[ \bigl(X_{t}^{p,k}
-X_{t} \bigr)^{2} \bigr] \leq 2\underbrace{\sum
_{n=p+1}^{\infty }~\sum_{|\alpha |=n}
x_{\alpha
}^{2}(t)}_{=:A_{1}(p)} + 2\underbrace{\sum
_{l=k+1}^{\infty }~\sum
_{n=0} ^{p}~\sum_{\substack{|\alpha |=n\\[0.1em]
d(\alpha )=l}}
{x_{\alpha }} ^{2}(t)}_{=:A_{2}(p,k)},
\end{align}
where $d(\alpha ):= \max \{i\geq 1:~\alpha _{i}>0\}$. To determine an
upper bound for $A_{1}(p)$, we use the chaos expansions\index{chaos expansion}
\eqref{choasx} and \eqref{Wienerchaos2} of $X_{t}$, i.e.,
\begin{align*}
X_{t} = \sum_{n=0}^{\infty }~
\sum_{|\alpha |=n} x_{\alpha }(t) \varPsi
^{\alpha } = \sum_{n=0}^{\infty }I_{n}
\bigl(\xi _{n}\bigl(\mathbf{t}^{n}; t\bigr)\bigr),
\end{align*}
with $\mathbf{t}^{n} = (t_{1},\ldots ,t_{n})$ and symmetric kernel
functions $\xi _{n}(\mathbf{t}^{n};t)$ being defined by
\begin{equation*}
\xi _{n}(\cdot ;t) = \frac{1}{n!}\mathbh{E} \bigl[
D^{n} X_{t} \bigr]
\end{equation*}
for all $n\in \mathbh{N}_{0}$. We conclude that
%
\begin{align}
\label{x2identity} \sum_{|\alpha |=n} {x_{\alpha }}^{2}(t)
&= \mathbh{E} \biggl[ \biggl(\sum_{|\alpha |=n}
x_{\alpha }(t)\varPsi ^{\alpha } \biggr)^{2} \biggr] =
\mathbh{E} \bigl[ \bigl(I_{n}\bigl(\xi _{n}\bigl(
\mathbf{t}^{n};t\bigr)\bigr) \bigr)^{2} \bigr]
\nonumber
\\[1.5 ex] &= n! \bigl\langle \xi _{n}\bigl(\mathbf{t}^{n};
\,t\bigr),\xi _{n}\bigl(\mathbf{t} ^{n};\,t\bigr) \bigr
\rangle _{\mathbh{H}}
\nonumber
\\[1.5 ex] &= (n!)^{2} \int^{(n);t} \bigl(\xi
_{n}\bigl(\mathbf{t}^{n};\,t\bigr) \bigr)
^{2} d\mathbf{t}^{n},
\end{align}
where we abbreviate
\begin{align*}
\int^{(n);t} f(\cdot ) d\mathbf{t}^{n} :=\int
_{0}^{t} \ldots \int_{0}
^{t_{2}} f(\cdot ) dt_{1} \ldots dt_{n}.
\end{align*}
Therefore, we deduce via \eqref{x2identity} that
%
\begin{align}
\label{A1bound} A_{1}(p) = \sum_{n=p+1}^{\infty }(n!)^{2}
\int^{(n);t} \bigl(\xi _{n}\bigl(
\mathbf{t}^{n};t\bigr) \bigr)^{2} d\mathbf{t}^{n}
= \sum_{n=p+1}^{\infty
}\int
^{(n);t} \mathbh{E} \bigl[ \bigl(D^{n}_{t_{1},\ldots ,t_{n}}X
_{t} \bigr)^{2} \bigr] d\mathbf{t}^{n}.
\end{align}
Finally, using \eqref{A1bound} and applying Proposition \ref{prop1} we
obtain
%
\begin{align}
\label{A1estimate} A_{1}(p) &= \sum_{n=p+1}^{\infty }
\int^{(n);t} \mathbh{E} \bigl[ \bigl(D^{n}_{t_{1},\ldots ,t_{n}}X_{t}
\bigr)^{2} \bigr] d \mathbf{t}^{n}
\nonumber
\\ &\leq \bigl(1+{x_{0}}^{2} \bigr) \sum
_{n=p+1}^{\infty }C^{n} \int
^{(n);t} d\mathbf{t}^{n}
\nonumber
\\ &= \bigl(1+{x_{0}}^{2} \bigr) \sum
_{n=p+1}^{\infty } \frac{(Ct)^{n}}{n!}
\nonumber
\\ &\leq C \bigl(1+{x_{0}}^{2} \bigr)
\frac{1}{(p+1)!}.
\end{align}
The treatment of the term $A_{2}(p,k)$ is more involved. Recalling the
definition of $d(\alpha )= \max \{i\geq 1:~\alpha _{i}>0\}$, we introduce
the characteristic set $(i_{1}, \ldots , i_{n})$ of $\alpha \in
\mathcal{I}$ for $i_{1}\leq i_{2} \leq \cdots \leq i_{n}$ and
$|\alpha |=n$. It is defined as follows: $i_{1}$ is the index of the
first non-zero component of $\alpha $. If $\alpha _{i_{1}}=1$ then
$i_{2}$ is the index of the second non-zero component of $\alpha $;
otherwise $i_{1}=i_{2}=\cdots =i_{\alpha _{i_{1}}}$ and $i_{\alpha _{i
_{1}}+1}$ is the second non-zero component of $\alpha $. The same
operation is repeated for the index $i_{\alpha _{i_{1}}+1}$ and further
non-zero components of $\alpha $. In this fashion, the characteristic
set is constructed, resulting in the observation that $d(\alpha )=i
_{n}$. To give a simple example, consider the multiindex $\alpha =(2,0,1,4,0,0,
\ldots )\in \mathcal{I}$. Here the characteristic set of $\alpha $ is
given by
\begin{equation*}
(1,1,3,4,4,4,4).
\end{equation*}
For any $\alpha \in \mathcal{I}$ with $|\alpha |=n$ and a basis
$(e_{i})_{i\geq 1}$ of $\mathbh{H}=L^{2}([0,T])$, we denote by
$\widetilde{e}_{\alpha }(\mathbf{t}^{n})$ the (scaled) symmetrised form
of $\bigotimes e_{i}^{\otimes \alpha _{i}}$ defined via
%
\begin{equation}
\label{symme} \widetilde{e}_{\alpha }\bigl(\mathbf{t}^{n}
\bigr) := \sum_{\pi \in \mathfrak{P}
^{n}} e_{i_{1}}
(t_{\pi (1)} ) \cdots e_{i_{n}} (t _{\pi (n)} ),
\end{equation}
where $(i_{1}, \ldots , i_{n})$ is the characteristic set of
$\alpha $ and the sum runs over all permutations $\pi $ within the
permutation group $\mathfrak{P}^{n}$ of $\{1,\ldots ,n\}$. From
\eqref{xformula} we know that $x_{\alpha } (t) = \mathbh{E}  [ X
_{t} \varPsi ^{\alpha }  ]$. On the other hand, we have the
representation
\begin{equation*}
\varPsi ^{\alpha } = \frac{1}{\sqrt{\alpha !}} I_{ \llvert \alpha  \rrvert } \Biggl(
\bigotimes_{i=1} ^{\infty }e_{i}^{\otimes \alpha _{i}}
\Biggr) = \frac{1}{\sqrt{
\alpha !}} I_{ \llvert \alpha  \rrvert } \biggl( \frac{1}{ \llvert \alpha  \rrvert !}
\widetilde{e} _{\alpha } \biggr).
\end{equation*}
Now, for any $\alpha \in \mathcal{I}$ with $|\alpha |=n$, we obtain the
identity
%
\begin{align}
\label{xalphaformula} x_{\alpha }(t) = \frac{n!}{\sqrt{\alpha !}} \int
^{(n);t} \xi _{n}\bigl( \mathbf{t}^{n};t
\bigr) \widetilde{e}_{\alpha }\bigl(\mathbf{t}^{n}\bigr) d
\xch{\mathbf{t} ^{n},}{\mathbf{t} ^{n}.}
\end{align}
by \eqref{IBP}. Since
\begin{equation*}
\widetilde{e}_{\alpha }\bigl(\mathbf{t}^{n}\bigr)=\sum
_{j=1}^{n} e_{i_{n}}(t
_{j})\cdot \widetilde{e}_{\alpha ^{-}(i_{n})}\bigl(
\mathbf{t}_{j}^{n}\bigr),
\end{equation*}
where $\mathbf{t}_{j}^{n}$ is obtained from $\mathbf{t}^{n}$ by omitting
$t_{j}$ and $\alpha ^{-}(\cdot )$ denotes the diminished multi-index as
defined in \eqref{alphaminus}, we deduce that
%
\begin{align}
\label{xalphaidentity} x_{\alpha }(t) = \frac{n!}{\sqrt{\alpha !}} \sum
_{j=1}^{n}~ \int^{(n-1);t}
\Biggl( \int_{t_{j-1}}^{t_{j+1}} \xi _{n}
\bigl(\mathbf{t} ^{n};t\bigr)\,e_{i_{n}}(t_{j})
dt_{j} \Biggr) \widetilde{e}_{\alpha ^{-}(i
_{n})}\bigl(
\mathbf{t}_{j}^{n}\bigr) d\mathbf{t}_{j}^{n},
\end{align}
with $t_{0} := 0$ and $t_{n+1} := t$, by changing the order of
integration. Then for any $i_{n}=l\geq 1$ we integrate by parts to
deduce
%
\begin{align}
\label{int-by-parts-inner-part} \int_{t_{j-1}}^{t_{j+1}} \xi
_{n}\bigl(\mathbf{t}^{n};t\bigr)\,e_{l}(t_{j})dt
_{j} = \bigl[ \xi _{n}\bigl(\mathbf{t}^{n};t
\bigr)\,E_{l}(t_{j}) \bigr]_{t
_{j}=t_{j-1}}^{t_{j}=t_{j+1}}
- \int_{t_{j-1}}^{t_{j+1}} \frac{\partial
}{\partial t_{j}}\xi
_{n}\bigl(\mathbf{t}^{n};t\bigr)\,E_{l}(t_{j})dt_{j},
\end{align}
where
\begin{equation*}
E_{i}(s) = \int_{0}^{s}
e_{i}(u)du .
\end{equation*}
Now, we use substitution in \eqref{xalphaidentity} by renaming
$\mathbf{t}_{j}^{n}$ in the following way for each $j$: With
$s_{i}=t_{i}$ for all $i\leq j-1$ and $s_{i}=t_{i+1}$ for all
$i>j-1$, we have $\mathbf{s}^{n-1}:= \mathbf{t}_{j}^{n}$ by setting
$s_{0}=0$ and $s_{n}=t$. Moreover, we denote with $\mathbf{s}^{n-1,r}$,
$r=1,\ldots ,n-1$, the set that is generated from $\mathbf{s}^{n-1}$ by
taking $s_{r}$ twice. To finalize this notation, we set $\mathbf{s}
^{n-1,0} = (s_{0},s_{1},\ldots ,s_{n-1})$ and $\mathbf{s}^{n-1,n} = (s
_{1},\ldots ,s_{n-1},s_{n})$. Then
\begin{equation*}
\bigl[ \xi _{n}\bigl(\mathbf{t}^{n};t\bigr)
\,E_{l}(t_{j}) \bigr]_{t_{j}=t_{j-1}}
^{t_{j}=t_{j+1}} = \xi _{n}\bigl(\mathbf{s}^{n-1,j}\bigr)
\,E_{l}(s_{j}) - \xi _{n}\bigl(
\mathbf{s}^{n-1,j-1}\bigr)\,E_{l}(s_{j-1}), \ \
j=1,\ldots ,n.
\end{equation*}
Because $E_{l}(s_{0}) = E_{l}(0) = 0$ and $E_{l}(s_{n}) = E_{l}(t)$,
from \eqref{int-by-parts-inner-part} we see that by summing over $j$ all
terms except one cancel out. Hence, setting
\begin{align*}
\psi _{l}\bigl(\mathbf{s}^{n-1};t\bigr) &:= \xi
_{n}\bigl(\mathbf{s}^{n-1,n}\bigr)\,E_{l}(t)
\\[1.5 ex] & \qquad - \int_{0}^{s_{1}}
\frac{\partial }{\partial s_{1}}\xi _{n}\bigl(\tau , \mathbf{s}^{n-1};t
\bigr)\,E_{l}(\tau ) d\tau
\\[1.5 ex] & \qquad - \sum_{j=2}^{n-1}
\int_{s_{j-1}}^{s_{j}} \frac{\partial }{\partial s
_{j}}\xi
_{n}(\ldots ,s_{j-1},\tau ,s_{j+1},\ldots ;t)
\,E_{l}(\tau ) d \tau
\\[1.5 ex] & \qquad - \int_{s_{n-1}}^{t}
\frac{\partial }{ \partial s_{n}}\xi _{n}\bigl( \mathbf{s}^{n-1},\tau ;t
\bigr)\,E_{l}(\tau ) d\tau ,
\end{align*}
we obtain from \eqref{xalphaidentity}
\begin{align*}
\sum_{\substack{|\alpha |=n \\ i_{n} = d(\alpha ) = l}} {x_{\alpha }}
^{2}(t) &= \sum_{\substack{|\alpha |=n \\ i_{n} = d(\alpha ) = l}} \biggl(
\frac{n!}{\sqrt{\alpha !}} \int^{(n-1); t} \psi _{l}\bigl(
\mathbf{s}^{n-1};t\bigr)\,\widetilde{e}_{\alpha ^{-}(l)}\bigl(
\mathbf{s}^{n-1}\bigr) d \mathbf{s}^{n-1}
\biggr)^{2}
\\[1.5 ex] &\leq n^{2}\,\sum_{|\beta |=n-1}
\biggl( \frac{(n-1)!}{ \sqrt{
\beta !}} \int^{(n-1); t} \psi
_{l}\bigl(\mathbf{s}^{n-1};t\bigr)\,\widetilde{e}
_{\beta }\bigl(\mathbf{s}^{n-1}\bigr)d\mathbf{s}^{n-1}
\biggr)^{2},
\end{align*}
since $ \llvert {\alpha ^{-}(i_{|\alpha |}}) \rrvert  = |\alpha |-1$ and
$\alpha !\geq \alpha ^{-}(i_{|\alpha |})!$. In order to interpret the
last sum, consider the random variable $I_{n-1}   (\psi _{l}(
\mathbf{s}^{n-1};t)   )$. Then, following \eqref{x2identity} and
the identity \eqref{xalphaformula}, we conclude that
\begin{align*}
\mathbh{E} \bigl[ \bigl(I_{n-1} \bigl(\psi _{l}\bigl(
\mathbf{s}^{n-1};t\bigr) \bigr) \bigr)^{2} \bigr] &=
\bigl((n-1)!\bigr)^{2} \int^{(n-1);t} \psi
_{l} ^{2}\bigl(\mathbf{s}^{n-1};t\bigr) d
\mathbf{s}^{n-1}
\\[1.5 ex] &= \sum_{|\beta |=n-1}\! \biggl(
\frac{(n-1)!}{ \sqrt{\beta !}}\! \int^{(n-1); t}\! \psi _{l}\bigl(
\mathbf{s}^{n-1};t\bigr)\,\widetilde{e}_{\beta }\bigl(
\mathbf{s}^{n-1}\bigr)d\mathbf{s}^{n-1}
\biggr)^{2}.
\end{align*}
Thus, we get the estimate
\begin{equation*}
\sum_{\substack{|\alpha |=n
\\
i_{n} = l}} {x_{\alpha }}^{2}(t)
\leq (n!)^{2} \int^{(n-1);t} \psi
_{l}^{2}\bigl(\mathbf{s}^{n-1};t\bigr) d
\mathbf{s}^{n-1}.
\end{equation*}
Now, using the Cauchy--Schwarz inequality we deduce that
\begin{align*}
\psi _{l}^{2}\bigl(\mathbf{s}^{n-1};t\bigr)
&\leq (n+1) \Biggl(\xi _{n}^{2}\bigl(
\mathbf{s}^{n-1,n}\bigr)\,E_{l}^{2}(t)
\\[1.5 ex] &\qquad  + \int_{0}^{t} E_{l}^{2}(
\tau ) d\tau \cdot \sum_{j=1}^{n} \int
_{s_{j-1}}^{s_{j}} \biggl\llvert \frac{\partial }{\partial s_{j}}
\xi _{n}(\ldots s_{j-1},\tau , s_{j+1}, \ldots
;t ) \biggr\rrvert ^{2} d\tau \Biggr).
\end{align*}
Recall the identity $\xi _{n}(\cdot ;t) = (n!)^{-1}\mathbh{E}  [ D
^{n} X_{t}  ]$. Using similar arguments as in the proof of
Proposition \ref{prop1}, we obtain the inequality
\begin{equation*}
\biggl\llvert \frac{\partial }{\partial t_{i}}\xi _{n}(t_{1},\ldots
,t_{n};\,t) \biggr\rrvert ^{2} \leq \bigl(1+x_{0}^{2}
\bigr) \frac{C^{n}}{(n!)^{2}} , \quad i=1,\ldots ,n.
\end{equation*}
Since $|s_{j}-s_{j-1}|\leq t$, we finally conclude that
%
\begin{align}
\label{finalestimate} \psi _{l}^{2}\bigl(
\mathbf{s}^{n-1};t\bigr) \leq \bigl(1+x_{0}^{2}
\bigr) \frac{C^{n}(n+1)}{(n!)^{2}} \Biggl( E_{l}^{2}(t) + \int
_{0}^{t} E_{l} ^{2}(
\tau ) d\tau \Biggr).
\end{align}
Consequently, we deduce
\begin{align*}
\sum_{\substack{|\alpha |=n \\ i_{n}=l}} {x_{\alpha }}^{2}(t)
\leq \bigl(1+x _{0}^{2}\bigr) \frac{C^{n}(n+1)}{(n-1)!}
\Biggl( E_{l}^{2}(t) + \int_{0}
^{t} E_{l}^{2}(\tau ) d\tau \Biggr).
\end{align*}
Putting things together we conclude
%
\begin{align}\nonumber
A_{2}(p,k) &= \sum_{l=k+1}^{\infty }~
\sum_{n=1}^{p}~ \sum
_{\substack{|\alpha |=n \\ i_{n}=l}} {x_{\alpha }}^{2}(t)
\\[1.5 ex] &\leq \bigl(1+x_{0}^{2}\bigr) \sum
_{n=1}^{\infty } \frac{C^{n}(n+1)}{(n-1)!} \cdot \sum
_{l=k+1}^{\infty } \Biggl( E_{l}^{2}(t)
+ \int_{0}^{t} E _{l}^{2}(
\tau ) d\tau \Biggr)
\nonumber
\\[1.5 ex] &\leq C \bigl(1+x_{0}^{2}\bigr) \sum
_{l=k+1}^{\infty } \Biggl( E_{l}^{2}(t)
+ \int_{0}^{t} E_{l}^{2}(
\tau ) d\tau \Biggr).\label{A2estimate}
\end{align}
Combining \eqref{A1estimate} and \eqref{A2estimate} we obtain the result
of Theorem \ref{th2}.

%
\begin{figure}[t!]
\includegraphics{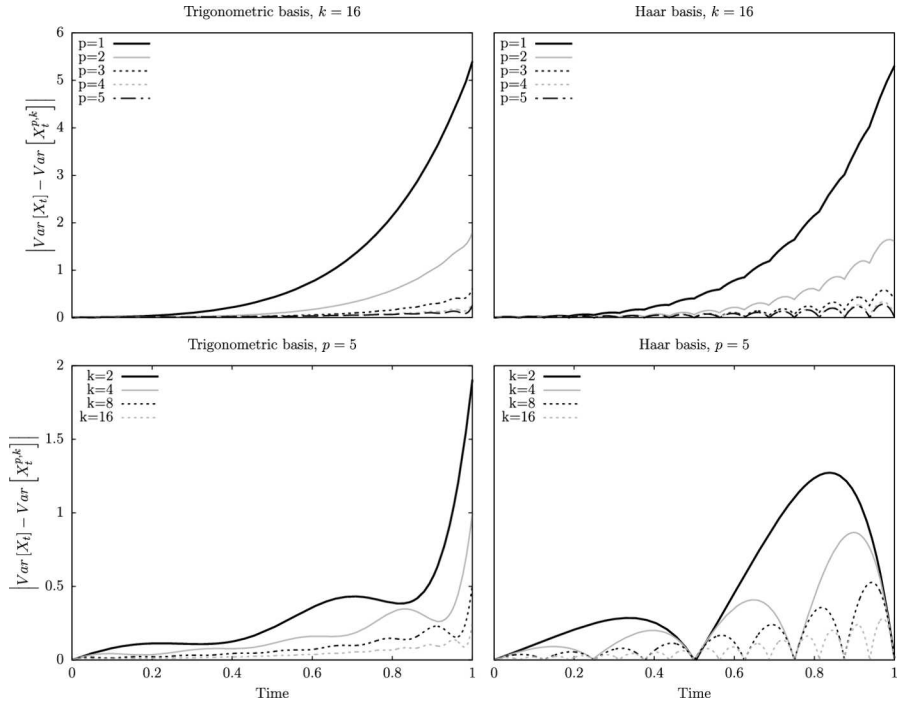}
\caption{Absolute errors of the variance as calculated with different
polynomial chaos expansions for the Geometric Brownian Motion,
$ \llvert  Var  [X_{t}  ] - Var  [X^{p,k}_{t}  ]
 \rrvert $. The trigonometric and Haar bases (with
$k=2^{n}$) from Example~\ref{ex} are shown}\label{fig1}
\end{figure}

\section{Numerical results and sparse truncation\index{sparse truncation}}%
\label{sec5}
To get an idea of the performance of the polynomial chaos expansion\index{polynomial chaos expansions} for
different bases and truncations, we apply the trigonometric basis\index{trigonometric basis} from
Example~\ref{ex} (i) and the Haar basis\index{Haar basis} from Example~\ref{ex} (ii) to
the geometric Brownian motion from Example~\ref{ex32},
%
\begin{align}
d X_{t} = \mu X_{t} dt + \sigma X_{t}
dW_{t}, \quad X_{0} = x_{0}
\end{align}
on the horizon $t \in [0,1]$ with $\mu =\sigma =x_{0}=1$. In this
setting we have the identity
\begin{equation*}
\text{Var}(X_{t}) = {x_{0}}^{2} \exp (2\mu
t ) \bigl(\exp \bigl(\sigma ^{2} t \bigr) - 1 \bigr),\vadjust{\goodbreak}
\end{equation*}
which can be used to compare the absolute errors of approximated
variances using different bases and values for $k$ and $p$. We used the
\texttt{ode45} solver with standard settings within
\texttt{Matlab 2017b} (64bit, single core) on an Intel Core i5-6300U
(2.4GHz, 8GB RAM).

\subsection{Absolute errors of variances}

Figure~\ref{fig1} shows the absolute errors between the variances of the
analytical solution and different polynomial chaos expansions.\index{polynomial chaos expansions}
One observes the fast decay of the approximation error for fixed $k$ and
increasing $p$, and vice versa. The main difference between the observed
errors is that the maximum of the trigonometric expansion is always at
the end of the time horizon at $t=1$, while the Haar basis\index{Haar basis} leads to
almost zero error on $k$ equidistant time points, including $t=0$ and
$t=1$.

Table~\ref{tab1} in the Appendix lists approximation errors, numbers of
coefficients, and computational times for several choices of $p$ and
$k$. The exponential increase in the number of coefficients with
increasing $p$ and $k$ leads to an exponential increase in runtime, as
expected. To overcome this issue, we propose a heuristic to reduce the
computational burden.

\subsection{Sparse truncation\index{sparse truncation}}%
\label{secsparse}
The information contained in the coefficient functions $x_{\alpha }(
\cdot )$ decays with increasing order $p$ of the basis polynomials
$\varPsi ^{\alpha }$ and the decaying rate of the Gaussian expansion, i.e.,
the index $k$ of the basis polynomials $e_{k}$ used for constructing
$\varPsi ^{\alpha }$ \cite{L06}.

Hence, we can define \emph{sparse index sets} for reducing the number
of multi-indices $\alpha $ from the full truncated index set
$\mathcal{I}_{p,k}$ without losing too much information, i.e., accuracy
in a numerical simulation.

\begin{definition}[Sparse truncation\index{sparse truncation} of first order]
\label{dfn:first-sparse-truncation}
Let $p$ be the maximum order of the index
$\alpha \in \mathcal{I}_{p,k}$. Then the
\emph{first order sparse index} $\mathbh{r}=(r_{1},\ldots ,r_{k})$
satisfies $p = r_{1}\geq r_{2}\geq \cdots \geq r_{k}$ and we define the
\emph{first order sparse index set}
%
\begin{equation}
\label{eq:first-sparse-index-set} \mathcal{I}_{p,k}^{\mathbh{r}} = \{ \alpha \in
\mathcal{I}_{p,k} \mid\alpha _{i}\leq
r_{i}\,\forall i\leq k \}.
\end{equation}
\end{definition}

\begin{ex}
Let $k=5$ and $p=3$. Then a possible choice of the sparse index is
$\mathbh{r}=(3,2,2,1,1)$. For constructing the first order polynomials
all five first order basis polynomials can be used. The second order
polynomials are comprised by all possible combinations of first order
basis polynomials depending on $e_{1},\ldots ,e_{5}$ and the second
order basis polynomials of $e_{1},e_{2},e_{3}$. Analogously, the third
order polynomials are constructed. By using this first order sparse
index set $\mathcal{I}_{p,k}^{\mathbh{r}}$ the number of coefficient
functions $x_{\alpha }(\cdot )$ appearing within the propagator\index{propagator system} system
\eqref{propagator} can be reduced drastically without impairing the
solution much. The full index set $\mathcal{I}_{p,k}$ consists of
$\frac{(k+p)!}{k!p!}=56$ terms, whereas this first order sparse index
set includes $42$ terms.
\end{ex}

An even more rigorous reduction of the number of coefficients included
in the propagator\index{propagator system} system can be achieved by using a \emph{second order
sparse index}, i.e., a series of first order sparse indices
$  (\mathbh{r}^{j}  )_{j=0,\ldots ,p}$ that depend on the actual
order of the polynomials $\varPsi ^{\alpha }$ with $j=\vert \alpha \vert
$, \cite{L06}. This approach allows to exclude crossing products of
random variables $W(e_{i})$ from the construction of higher order basis
polynomials $\varPsi ^{\alpha }$ that add only negligible information to the
system.

\begin{definition}[Sparse truncation\index{sparse truncation} of second order]
\label{dfn:second-sparse-truncation}
Let $p$ be the maximum order of the index
$\alpha \in \mathcal{I}_{p,k}$. Then the
\emph{second order sparse index} $(\mathbh{r}) =   (\mathbh{r}^{j}
  )_{j\leq p}$ is a series of first order sparse indices
$\mathbh{r}^{j}=  (r_{1}^{j},\ldots ,r_{k}^{j}  )$ satisfying
$j=r_{1}^{j}\geq r_{2}^{j}\geq \cdots \geq r_{k}^{j}$ for all
$j=|\alpha |\leq p$ and we define the \emph{second order sparse index
set}
%
\begin{equation}
\label{eq:second sparse-index-set} \mathcal{I}_{p,k}^{(\mathbh{r})} = \bigl\{ \alpha
\in \mathcal{I} _{p,k} \mid\alpha _{i}\leq
r_{i}^{j}\,\forall i\leq k,\,\forall j\leq p \bigr\}.
\end{equation}
\end{definition}

\begin{ex}
Considering again the setting of the previous example with $k=5$ and
$p=3$, one possible choice of a second order sparse index is given via
$\mathbh{r}^{1}=(1,1,1,1,1)$, $\mathbh{r}^{2}=(2,2,2,1,0)$, and
$\mathbh{r}^{3}=(3,2,0,0,0)$. In constructing basis polynomials of order
$\vert \alpha \vert =3$ we can use all combinations of basis polynomials
depending on the first two random variables $e_{1}$ and $e_{2}$ up to
orders $3$ and $2$, respectively. Thus, these are $\sqrt{6}H_{3}(e
_{1})$, $\sqrt{2}H_{2}(e_{1})H_{1}(e_{2})$, and $\sqrt{2}H_{1}(e
_{1})H_{2}(e_{2})$, compare \cite{L06}.
\end{ex}

Table~\ref{tab:indices} in the Appendix lists the first and second order
sparse indices that were used for the numerical study. The results in
Table~\ref{tab1} show that at the price of a slightly higher error, and
of course the loss of a guaranteed upper bound as in \eqref{estimate},
the computational times could be reduced by several orders of magnitude.

\begin{appendix}
\section*{Appendix}

\begin{table}[h!]
\caption{Computational times in seconds and error estimates for
trigonometric and Haar bases from Example~\ref{ex} and different values
of $p$ and $k$. The number of coefficients $n_{\text{coeff}}$ depends
also on the type of truncation, compare Table~\ref{tab:indices} for
details. error$_{t=1}$ is the absolute error of the expansion's variance
$| Var  [X_{t}  ] - Var  [X^{p,k}_{t}  ] |$ at $t=1$
while error$_{\text{max}}$ is the maximum value over $t \in [0,1]$}
\label{tab1}
\begin{tabular*}{\textwidth}{@{\extracolsep{\fill}}rrrr| rrr | rrr}
\hline
\multicolumn{4}{c|}{} &
\multicolumn{3}{c|}{Trigonometric basis} &
\multicolumn{3}{c}{Haar basis} \\
$k$ &
$p$ &
$n_{\text{coeff}}$ &
type &
time &
$\text{error}_{t=1}$ &
$\text{error}_{\text{max}}$ &
time &
$\text{error}_{t=1}$ &
$\text{error}_{\text{max}}$\\
\hline
2  & 1 & 3     & full        & 0.04      & 6.04 & 6.04 & 0.03      & 5.31 & 5.31 \\
4  & 1 & 5     & full        & 0.02      & 5.68 & 5.68 & 0.02      & 5.31 & 5.31 \\
8  & 1 & 9     & full        & 0.22      & 5.49 & 5.49 & 0.04      & 5.30 & 5.30 \\
16 & 1 & 17    & full        & 0.05      & 5.40 & 5.40 & 0.19      & 5.31 & 5.31 \\
32 & 1 & 33    & full        & 0.17      & 5.35 & 5.35 & 1.31      & 5.30 & 5.30 \\
64 & 1 & 65    & full        & 1.28      & 5.33 & 5.33 & 10.12     & 5.31 & 5.31 \\
2  & 2 & 6     & full        & 0.02      & 3.04 & 3.04 & 0.02      & 1.61 & 1.83 \\
4  & 2 & 15    & full        & 0.03      & 2.35 & 2.35 & 0.05      & 1.61 & 1.76 \\
8  & 2 & 45    & full        & 0.20      & 1.98 & 1.98 & 0.82      & 1.61 & 1.69 \\
8  & 2 & 41    & \sparse{1}  & 0.31      & 1.99 & 1.99 & 0.24      & 1.61 & 1.69 \\
8  & 2 & 19    & \sparse{2}  & 0.04      & 2.16 & 2.16 & 0.07      & 1.61 & 1.72 \\
16 & 2 & 153   & full        & 3.85      & 1.80 & 1.80 & 18.61     & 1.61 & 1.65 \\
16 & 2 & 141   & \sparse{3}  & 1.00      & 1.80 & 1.80 & 4.60      & 1.61 & 1.65 \\
16 & 2 & 27    & \sparse{4}  & 0.05      & 2.07 & 2.07 & 0.19      & 1.61 & 1.67 \\
32 & 2 & 561   & full        & 86.83     & 1.71 & 1.71 & 554.61    & 1.61 & 1.63 \\
32 & 2 & 537   & \sparse{5}  & 22.80     & 1.71 & 1.71 & 143.90    & 1.61 & 1.63 \\
32 & 2 & 69    & \sparse{6}  & 0.31      & 1.84 & 1.84 & 3.22      & 1.61 & 1.64 \\
64 & 2 & 2145  & full        & 2189.70   & 1.66 & 1.66 & 17234.51  & 1.61 & 1.62 \\
2  & 3 & 10    & full        & 0.02      & 2.15 & 2.15 & 0.02      & 0.38 & 1.35 \\
4  & 3 & 35    & full        & 0.10      & 1.29 & 1.29 & 0.19      & 0.38 & 1.01 \\
8  & 3 & 165   & full        & 2.76      & 0.84 & 0.84 & 10.58     & 0.38 & 0.76 \\
8  & 3 & 127   & \sparse{7}  & 0.53      & 0.85 & 0.85 & 1.77      & 0.37 & 0.76 \\
8  & 3 & 37    & \sparse{8}  & 0.06      & 1.11 & 1.11 & 0.13      & 0.38 & 0.86 \\
16 & 3 & 969   & full        & 200.34    & 0.61 & 0.61 & 1070.17   & 0.38 & 0.58 \\
16 & 3 & 763   & \sparse{9}  & 45.61     & 0.62 & 0.62 & 169.16    & 0.38 & 0.59 \\
16 & 3 & 45    & \sparse{10} & 0.14      & 1.02 & 1.02 & 0.39      & 0.38 & 0.76 \\
2  & 4 & 15    & full        & 0.02      & 1.94 & 1.94 & 0.02      & 0.07 & 1.32 \\
4  & 4 & 70    & full        & 0.30      & 1.04 & 1.04 & 0.79      & 0.07 & 0.87 \\
8  & 4 & 495   & full        & 26.99     & 0.57 & 0.57 & 117.55    & 0.07 & 0.56 \\
8  & 4 & 303   & \sparse{11} & 4.34      & 0.57 & 0.57 & 12.30     & 0.07 & 0.52 \\
8  & 4 & 32    & \sparse{12} & 0.05      & 0.96 & 0.96 & 0.12      & 0.07 & 0.75 \\
16 & 4 & 4845  & full        & 5986.44   & 0.32 & 0.32 & 30591.94  & 0.07 & 0.33 \\
16 & 4 & 40    & \sparse{13} & 0.08      & 0.87 & 0.87 & 0.36      & 0.07 & 0.67 \\
32 & 4 & 92    & \sparse{14} & 0.50      & 0.59 & 0.59 & 3.56      & 0.07 & 0.45 \\
2  & 5 & 21    & full        & 0.03      & 1.91 & 1.91 & 0.03      & 0.01 & 1.27 \\
4  & 5 & 126   & full        & 0.92      & 1.00 & 1.00 & 2.04      & 0.01 & 0.87 \\
8  & 5 & 1287  & full        & 192.98    & 0.51 & 0.51 & 855.07    & 0.01 & 0.53 \\
8  & 5 & 599   & \sparse{15} & 10.04     & 0.51 & 0.51 & 50.03     & 0.01 & 0.55 \\
8  & 5 & 36    & \sparse{16} & 0.03      & 0.92 & 0.92 & 0.11      & 0.01 & 0.74 \\
16 & 5 & 20349 & full        & 120469.96 & 0.26 & 0.26 & 600591.17 & 0.01 & 0.28 \\
16 & 5 & 44    & \sparse{17} & 0.06      & 0.83 & 0.83 & 0.28      & 0.01 & 0.65 \\
32 & 5 & 98    & \sparse{18} & 0.51      & 0.55 & 0.55 & 3.60      & 0.01 & 0.42 \\
\hline
\end{tabular*}
\end{table}

\begin{table}
\caption{List of first and second order sparse indices used in Section~\ref{secsparse}
together with the number of resulting coefficient functions. The reference numbers
coincide with those in Table~\ref{tab1}}\label{tab:indices}
\begin{tabular*}{\textwidth}{@{\extracolsep{\fill}}cccclc}
\hline
symbol &
$p$ &
$k$ &
order &
\multicolumn{1}{c}{index $\mathbh{r}$/$(\mathbh{r})$} &
$n_{\text{coeff}}$ \\
\hline\rule{0pt}{10pt}
sp\tsup{1}  & 2 & 8  & 1 & $\phantom{\mbox{}^{1}}\mathbh{r}=(2,2,2,2,1,1,1,1)$             & 41  \\[2pt]
sp\tsup{2}  & 2 & 8  & 2 & $\mathbh{r}^{1}=(1,\ldots ,1)$                             & 19  \\
            &   &    &   & $\mathbh{r}^{2}=(2,2,2,2,0,0,0,0)$                         &     \\[2pt]
sp\tsup{3}  & 2 & 16 & 1 & $\phantom{\mbox{}^{1}}\mathbh{r}=(2,2,2,2,1,\ldots ,1)$         & 141 \\[2pt]
sp\tsup{4}  & 2 & 16 & 2 & $\mathbh{r}^{1}=(1,\ldots ,1)$                             & 27  \\
            &   &    &   & $\mathbh{r}^{2}=(2,2,2,2,0,\ldots ,0)$                     &     \\[2pt]
sp\tsup{5}  & 2 & 32 & 1 & $\phantom{\mbox{}^{1}}\mathbh{r}=(2,2,2,2,2,2,2,2,1,\ldots ,1)$ & 537 \\[2pt]
sp\tsup{6}  & 2 & 32 & 2 & $\mathbh{r}^{1}=(1,\ldots ,1)$                             & 69  \\
            &   &    &   & $\mathbh{r}^{2}=(2,2,2,2,2,2,2,2,0,\ldots ,0)$             &     \\
\hline\rule{0pt}{10pt}
sp\tsup{7}  & 3 & 8  & 1 & $\phantom{\mbox{}^{1}}\mathbh{r}=(3,3,2,2,1,1,1,1)$             & 127 \\[2pt]
sp\tsup{8}  & 3 & 8  & 2 & $\mathbh{r}^{1}=(1,\ldots ,1)$                             & 37  \\[2pt]
            &   &    &   & $\mathbh{r}^{2}=(2,2,2,2,0,0,0,0)$                         &     \\
            &   &    &   & $\mathbh{r}^{3}=(3,3,2,2,0,0,0,0)$                         &     \\[2pt]
sp\tsup{9}  & 3 & 16 & 1 & $\phantom{\mbox{}^{1}}\mathbh{r}=(3,3,2,2,1,\ldots ,1)$         & 763 \\[2pt]
sp\tsup{10} & 3 & 16 & 2 & $\mathbh{r}^{1}=(1,\ldots ,1)$                             & 45  \\
            &   &    &   & $\mathbh{r}^{2}=(2,2,2,2,0,\ldots ,0)$                     &     \\
            &   &    &   & $\mathbh{r}^{3}=(3,3,2,2,0,\ldots ,0)$                     &     \\
\hline\rule{0pt}{10pt}
sp\tsup{11} & 4 & 8  & 1 & $\phantom{\mbox{}^{1}}\mathbh{r}=(4,4,2,2,1,1,1,1)$             & 303 \\[2pt]
sp\tsup{12} & 4 & 8  & 2 & $\mathbh{r}^{1}=(1,\ldots ,1)$                             & 32  \\
            &   &    &   & $\mathbh{r}^{2}=(2,2,2,2,0,\ldots ,0)$                     &     \\
            &   &    &   & $\mathbh{r}^{3}=(3,3,2,,0,\ldots ,0)$                      &     \\
            &   &    &   & $\mathbh{r}^{4}=(4,3,0,\ldots ,0)$                         &     \\[2pt]
sp\tsup{13} & 4 & 16 & 2 & $\mathbh{r}^{1}=(1,\ldots ,1)$                             & 40  \\
            &   &    &   & $\mathbh{r}^{2}=(2,2,2,2,0,\ldots ,0)$                     &     \\
            &   &    &   & $\mathbh{r}^{3}=(3,3,2,0,\ldots ,0)$                       &     \\
            &   &    &   & $\mathbh{r}^{4}=(4,3,0,\ldots ,0)$                         &     \\[2pt]
sp\tsup{14} & 4 & 32 & 2 & $\mathbh{r}^{1}=(1,\ldots ,1)$                             & 92  \\
            &   &    &   & $\mathbh{r}^{2}=(2,2,2,2,2,2,2,2,0,\ldots ,0)$             &     \\
            &   &    &   & $\mathbh{r}^{3}=(3,3,2,2,0,\ldots ,0)$                     &     \\
            &   &    &   & $\mathbh{r}^{4}=(4,4,0,\ldots ,0)$                         &     \\
\hline\rule{0pt}{10pt}
sp\tsup{15} & 5 & 8  & 1 & $\phantom{\mbox{}^{1}}\mathbh{r}=(5,5,2,2,1,1,1,1)$             & 599 \\[2pt]
sp\tsup{16} & 5 & 8  & 2 & $\mathbh{r}^{1}=(1,\ldots ,1)$                             & 36  \\
            &   &    &   & $\mathbh{r}^{2}=(2,2,2,2,0,\ldots ,0)$                     &     \\
            &   &    &   & $\mathbh{r}^{3}=(3,3,2,,0,\ldots ,0)$                      &     \\
            &   &    &   & $\mathbh{r}^{4}=(4,3,0,\ldots ,0)$                         &     \\
            &   &    &   & $\mathbh{r}^{5}=(5,3,0,\ldots ,0)$                         &     \\[2pt]
sp\tsup{17} & 5 & 16 & 2 & $\mathbh{r}^{1}=(1,\ldots ,1)$                             & 44  \\
            &   &    &   & $\mathbh{r}^{2}=(2,2,2,2,0,\ldots ,0)$                     &     \\
            &   &    &   & $\mathbh{r}^{3}=(3,3,2,0,\ldots ,0)$                       &     \\
            &   &    &   & $\mathbh{r}^{4}=(4,3,0,\ldots ,0)$                         &     \\
            &   &    &   & $\mathbh{r}^{5}=(5,3,0,\ldots ,0)$                         &     \\[2pt]
sp\tsup{18} & 5 & 32 & 2 & $\mathbh{r}^{1}=(1,\ldots ,1)$                             & 98  \\
            &   &    &   & $\mathbh{r}^{2}=(2,2,2,2,2,2,2,2,0,\ldots ,0)$             &     \\
            &   &    &   & $\mathbh{r}^{3}=(3,3,2,2,0,\ldots ,0)$                     &     \\
            &   &    &   & $\mathbh{r}^{4}=(4,4,0,\ldots ,0)$                         &     \\
            &   &    &   & $\mathbh{r}^{5}=(5,5,0,\ldots ,0)$                         &     \\
\hline
\end{tabular*}%
\end{table}
\end{appendix}


\begin{funding}
This project has received funding from the \gsponsor[id=GS1,sponsor-id=100010663]{European Research Council}
(ERC) under the European Union's Horizon 2020 research and innovation
programme (grant agreement No \gnumber[refid=GS1]{647573}), from \gsponsor[id=GS4,sponsor-id=501100001659]{Deutsche Forschungsgemeinschaft} (DFG, German Research Foundation) -- \gnumber[refid=GS4]{314838170},
GRK 2297 MathCoRe, from the project ``Ambit fields: probabilistic
properties\index{probabilistic properties} and statistical inference'' funded by \gsponsor[id=GS7,sponsor-id=100008398]{Villum Fonden}, and from
CREATES funded by the \gsponsor[id=GS8,sponsor-id=501100001732]{Danish National Research Foundation}.
\end{funding}

%

\end{document}